\documentclass[12pt]{amsart}
\usepackage[active]{srcltx}
\usepackage{calc,amssymb,amsthm,amsmath,amscd, eucal,ulem}
\usepackage{alltt}
\usepackage[left=1.3in,top=1.25in,right=1.3in,bottom=1.25in]{geometry}
\RequirePackage[dvipsnames,usenames]{color}

\input{kmacros3.sty}
\input{mabliautoref.sty}
\input{xy}
\xyoption{all}
\usepackage{tikz}

\numberwithin{equation}{theorem}

\newcommand{\D}{\displaystyle}

\renewcommand{\m}{\mathfrak{m}}
\renewcommand{\n}{\mathfrak{n}}

\DeclareMathOperator{\depth}{depth}

\DeclareMathOperator{\sheafhom}{\scr{H}{\kern -2pt om}}
\DeclareMathOperator{\soc}{socle}
\DeclareMathOperator{\pd}{pd}
\usepackage{fullpage}

\usepackage{setspace}
\usepackage{hyperref}

\usepackage{enumerate}

\usepackage{graphicx}

\usepackage[all,cmtip]{xy}
\usepackage{verbatim}

\theoremstyle{theorem}

\renewcommand{\O}{\mathcal O}

\begin{document}
\title{Local cohomology of Du Bois singularities and applications to families}
\author{Linquan Ma}
\author{Karl Schwede}
\author{Kazuma Shimomoto}
\address{Department of Mathematics\\ University of Utah\\ Salt Lake City\\ UT 84112}
\email{lquanma@math.utah.edu}
\address{Department of Mathematics\\ University of Utah\\ Salt Lake City\\ UT 84112}
\email{schwede@math.utah.edu}
\address{Department of Mathematics College of Humanities and Sciences \\ Nihon University \\ Setagaya-ku \\Tokyo 156-8550 \\Japan}
\email{shimomotokazuma@gmail.com}

\thanks{The first named author was supported in part by the NSF Grant DMS \#1600198 and NSF CAREER Grant DMS \#1252860/1501102.  }
\thanks{The second named author was supported in part by the NSF FRG Grant DMS \#1265261/1501115 and NSF CAREER Grant DMS \#1252860/1501102}
\thanks{The third named author is partially supported by Grant-in-Aid for Young Scientists (B) \# 25800028.}


\maketitle

\begin{abstract}
In this paper we study the local cohomology modules of Du Bois singularities. Let $(R,\m)$ be a local ring, we prove that if $R_{\red}$ is Du Bois, then $H_\m^i(R)\to H_\m^i(R_{\red})$ is surjective for every $i$. We find many applications of this result. For example we answer a question of Kov\'acs and the second author \cite{KovacsSchwedeInversionofadjunctionforDB} on the Cohen-Macaulay property of Du Bois singularities. We obtain results on the injectivity of $\Ext$ that provide substantial partial answers of questions in \cite{EisenbudMustataStillmanCohomologyontoricvarieties} in characteristic $0$.  These results can also be viewed as generalizations of the Kodaira vanishing theorem for Cohen-Macaulay Du Bois varieties. We prove results on the set-theoretic Cohen-Macaulayness of the defining ideal of Du Bois singularities, which are characteristic $0$ analogs and generalizations of results of Singh-Walther and answer some of their questions in \cite{SinghandWaltherArithmeticrankofSegreproducts}. We extend results of Hochster-Roberts on the relation between Koszul cohomology and local cohomology for $F$-injective and Du Bois singularities \cite{HochsterRobertsFrobeniusLocalCohomology}. We also prove that singularities of dense $F$-injective type deform.
\end{abstract}

\section{Introduction}

The notion of Du~Bois singularities was introduced by Steenbrink based on work of Du~Bois \cite{DuBoisMain} which itself was an attempt to localize Deligne's Hodge theory of singular varieties \cite{DeligneHodgeIII}.  Steenbrink studied them initially because families with Du~Bois fibers have remarkable base-change properties \cite{SteenbrinkCohomologicallyInsignificant}.  On the other hand, Du~Bois singularities have recently found new connections across several areas of mathematics \cite{KollarKovacsLCImpliesDB,KovacsSchwedeDBDeforms,LeeLocalAcyclicFibrationsAndDeRham, CortinasHaesemeyerWalkerWeibelNegativeBass,HuberJorderDifferentialFormsHTopology}.  In this paper we find numerous applications of Du~Bois singularities especially to questions of local cohomology.  Our key observation is as follows.

\begin{keylemma*}[\autoref{surjectivity in local cohomology}]
Suppose $(S, \fram)$ is a local ring essentially of finite type over $\bC$ and that $S_{\red} = S/\sqrt{0}$ has Du~Bois singularities.  Then
\[
H_\m^i(S)\to H_\m^i(S_{\red})
\]
is surjective for every $i$.
\end{keylemma*}
In fact, we obtain this surjectivity for Du~Bois pairs, but we leave the statement simple in the introduction.  Utilizing this lemma and related methods, we prove several results.

\begin{theoremA*}[\autoref{cor.CMNearCartier}]
Let $X$ be a reduced scheme and let $H\subseteq X$ be a Cartier divisor. If $H$ is Du Bois and $X\backslash H$ is Cohen-Macaulay, then $X$ is Cohen-Macaulay and hence so is $H$.
\end{theoremA*}

This first consequence answers a question of Kov\'acs and the second author and should be viewed as a generalization of several results in \cite[Section 7]{KollarKovacsLCImpliesDB}.  In particular, we do not need a projective family.  This result holds if one replaces Cohen-Macaulay with \Sn{k} and also generalizes to the context of Du~Bois pairs as in \cite{KovacsSchwedeInversionofadjunctionforDB}.  Theorem A should also be viewed as a characteristic 0 analog of \cite[Proposition 2.13]{FedderWatanabe} and \cite[Corollary A.4]{HoriuchiMillerShimomoto}.  Theorem A also implies that if $X \setminus H$ has rational singularities and $H$ has Du~Bois singularities, then $X$ has rational singularities (see \autoref{cor.RatOutsidePlusDBinImpliesRat}), generalizing \cite[Theorem E]{KovacsSchwedeInversionofadjunctionforDB}.

Our next consequence of the key lemma is the following.
\begin{theoremB*}[\autoref{injectivity of Ext}]
Let $(R,\m)$ be a Gorenstein local ring essentially of finite type over $\mathbb{C}$, and let $I\subseteq R$ be an ideal such that $R/I$ has Du~Bois singularities. Then the natural map $\Ext^j_R(R/I, R)\to H_I^j(R)$ is injective for every $j$.
\end{theoremB*}
This gives a partial characteristic $0$ answer to \cite[Question 6.2]{EisenbudMustataStillmanCohomologyontoricvarieties}, asking when such maps are injective. A special case of the analogous characteristic $p>0$ result was obtained by Singh-Walther \cite{SinghandWaltherlocalcohomologyandpure}. This result leads to an answer of \cite[Questioin 7.5]{DeStefaniNunezBetancourtFthresholdGradedrings} on the bounds on the projective dimension of Du Bois and log canonical singularities.
In the graded setting, we can prove a much stronger result on the injectivity of Ext:
\begin{theoremC*}[\autoref{thm.injectivityofExtforpuncturedDB}]
Let $(R,\m)$ be a reduced Noetherian $\mathbb{N}$-graded $(R_0=\mathbb{C})$-algebra with $\m$ the unique homogeneous maximal ideal. Suppose $R_P$ is Du Bois for all $P\neq\m$. Write $R=A/I$ where $A=\mathbb{C}[x_1,\dots,x_n]$ is a polynomial ring with $\deg x_i=d_i>0$ and $I$ is a homogeneous ideal. Then the natural degree-preserving map $\Ext_A^j(R, A)\to H_I^j(A)$ induces an injection
\[
\big[\Ext_A^j(R, A)\big]_{\geq -d}\hookrightarrow \big[H_I^j(A)\big]_{\geq -d}
\] for every $j$, where $d=\sum d_i$.
\end{theoremC*}
Theorem C leads to new vanishing for local cohomology for $\mathbb{N}$-graded isolated non-Du Bois singularities \autoref{thm.vanishing of local cohomolog for DB}, a generalization of the Kodaira vanishing theorem for Cohen-Macaulay Du Bois projective varieties.

Our next consequence of the key lemma answers a longstanding question, but first we give some background.  Du~Bois singularities are closely related to the notion of $F$-injective singularities (rings where Frobenius acts injectively on local cohomology of maximal ideals).  In particular, it is conjectured that $X$ has Du~Bois singularities if and only if its reduction to characteristic $p > 0$ has $F$-injective singularities for a Zariski dense set of primes $p \gg 0$ (this is called \emph{dense $F$-injective type}).  This conjecture is expected to be quite difficult as it is closely related to asking for infinitely many ordinary characteristic $p > 0$ reductions for smooth projective varieties over $\bC$ \cite{BhattSchwedeTakagiweakordinaryconjectureandFsingularity}.  On the other hand, several recent papers have studied how Du~Bois and $F$-injective singularities deform \cite{KovacsSchwedeDBDeforms,HoriuchiMillerShimomoto}.  We know that if a Cartier divisor $H \subseteq X$ has Du~Bois singularities, then $X$ also has Du~Bois singularities near $H$.  However, the analogous statement for $F$-injective singularities in characteristic $p > 0$ is open and has been since \cite{FedderFPureRat}.  In fact $F$-injective singularities were introduced because it was observed that Cohen-Macaulay $F$-injective deform in this way but $F$-pure singularities\footnote{We now know that $F$-pure singularities are analogs of log canonical singularities \cite{HaraWatanabeFRegFPure}.} do not.  We show that at least the property of having dense $F$-injective type satisfies such a deformation result, in other words that $F$-injectivity deforms when $p$ is large.

\begin{theoremD*}[\autoref{thm.FInjectiveTypeDeforms}]
Let $(R,\m)$ be a local ring essentially of finite type over $\mathbb{C}$ and let $x$ be a nonzerodivisor on $R$. Suppose $R/xR$ has dense $F$-injective type. Then for infinitely many $p>0$, the Frobenius action $x^{p-1}F$ on $H_{\m_p}^i(R_p)$ is injective for every $i$, where $(R_p,\m_p)$ is the reduction mod $p$ of $R$. In particular, $R$ has dense $F$-injective type.
\end{theoremD*}

Our final result is a characteristic $0$ analog of a strengthening of the main result of \cite{HochsterRobertsFrobeniusLocalCohomology}, and we can give a characteristic $p > 0$ generalization as well.

\begin{theoremE*}[\autoref{analogue of Hochster-Roberts}]
Let $R$ be a Noetherian $\mathbb{N}$-graded $k$-algebra, with $\m$ the unique homogeneous maximal ideal. Suppose $R$ is equidimensional and Cohen-Macaulay on $\Spec R-\{\m\}$. Assume one of the following:
\begin{enumerate}
\item $k$ has characteristic $p>0$ and $R$ is $F$-injective.
\item $k$ has characteristic $0$ and $R$ is Du Bois.
\end{enumerate}
Then $\big[H^r(\underline{x}, R)\big]_0\cong H_\m^r(R)$ for every $r<n=\dim R$ and every homogeneous system of parameters $\underline{x}=x_1,\dots,x_n$, where $H^r(\underline{x}, R)$ denotes the $r$-th Koszul cohomology of $\underline{x}$. In other words, it is not necessary to take a direct limit when computing the local cohomology!
\end{theoremE*}
In fact, we prove a more general result \autoref{char free theorem in analog of Hochster-Roberts} which does not require any $F$-injective or Du Bois conditions, from which Theorem E follows immediately thanks to our injectivity and vanishing results \autoref{thm.injectivityofExtforpuncturedDB}, \autoref{thm.vanishing of local cohomolog for DB}.

\vskip 9pt
\noindent
{\it Acknowledgements:}  The authors are thankful to S\'andor Kov\'acs, Shunsuke Takagi and the referee for valuable discussions and comments on a previous draft of the paper.  We would like to thank Pham Hung Quy for discussions which motivate the proof of \autoref{inversion of adjunction for DB pair}. We also thank Zsolt Patakfalvi, Anurag K. Singh and Kevin Tucker for valuable discussions.


\section{Preliminaries}

Throughout this paper, all rings will be Noetherian and all schemes will be Noetherian and separated.  When in equal characteristic 0, we will work with rings and schemes that are essentially of finite type over $\bC$.  Of course, nearly all of our results also hold over all other fields of characteristic zero by base change.

\subsection{Du~Bois singularities}

We give a brief introduction to Du~Bois singularities and pairs.  For more details see for instance \cite{KovacsSchwedeDuBoisSurvey} and \cite{KollarKovacsSingularitiesBook}.  Frequently we will work in the setting of pairs in the Du~Bois sense.

\begin{definition}
Suppose $X$ and $Z$ are schemes of essentially finite type over $\bC$.  By a \emph{pair} we will mean the combined data of $(X, Z)$.  We will call the pair \emph{reduced} if both $X$ and $Z$ are reduced.
\end{definition}

Suppose that $X$ is a scheme essentially of finite type over $\bC$.  Associated to $X$ is an object $\DuBois{X} \in D^b_{\coherent}(X)$ with a map $\O_X \to \DuBois{X}$ functorial in the following way.  If $f : Y \to X$ is a map of schemes then we have a commutative square
\[
\xymatrix{
\O_X \ar[d] \ar[r] & \DuBois{X} \ar[d]\\
\myR f_* \O_Y \ar[r] & \myR f_* \DuBois{Y}\\
}
\]
If $X$ is non-reduced, then $\DuBois{X} = \DuBois{X_{\red}}$ by definition.  To define $\DuBois{X}$ in general let $\pi : X_{\mydot} \to X$ be a hyperresolution and set $\DuBois{X} = \myR \pi_* \O_{X_{\mydot}}$ (alternatively see \cite{SchwedeEasyCharacterization}).

If $Z \subseteq X$ is a closed subscheme, then we define $\DuBois{X,Z}$ as the object completing the following distinguished triangle
\[
\DuBois{X,Z} \to \DuBois{X} \to \DuBois{Z} \xrightarrow{+1}.
\]
We observe that there is an induced map $\sI_{Z \subseteq X} \to \DuBois{X,Z}$. We also note that by this definition and the fact $\DuBois{X} = \DuBois{X_{\red}}$, we have $\DuBois{X, Z} = \DuBois{X_{\red}, Z_{\red}}$.

\begin{definition}[Du~Bois singularities]
We say that $X$ has Du~Bois singularities if the map $\O_X \to \DuBois{X}$ is a quasi-isomorphism.  If $Z \subseteq X$ is a closed subscheme we say that $(X, Z)$ has Du~Bois singularities if the map $\sI_{Z \subseteq X} \to \DuBois{X,Z}$ is a quasi-isomorphism.
\end{definition}

\begin{remark}
It is clear that Du Bois singularities are reduced. In general, a pair $(X,Z)$ being Du Bois does not necessarily imply $X$ or $Z$ is reduced. However, if a pair $(X, Z)$ is Du Bois and $X$ is reduced, then so is $Z$ \cite[Lemma 2.9]{KovacsSchwedeInversionofadjunctionforDB}.
\end{remark}

\subsection{Cyclic covers of non-reduced schemes}

In this paper we will need to take cyclic covers of non-reduced schemes.  We will work in the following setting.  We assume what follows is well known to experts but we do not know a reference in the generality we need (see \cite[Section 2.4]{KollarMori} or \cite[Subsection 2.1.1]{deFernexEinMustataBook}).  Note also that the reason we are able to work with non-reduced schemes is because our $\sL$ is actually locally free and \emph{not} just a reflexive rank-1 sheaf.

\begin{setting}
\label{set.GeneralCyclicCover}
Suppose $X$ is a scheme of finite type over $\bC$ (typically projective).  Suppose also that $\sL$ is a line bundle (typically semi-ample).  Choose a (typically general) global section $s \in H^0(X, \sL^n)$ for some $n > 0$ (typically $n \gg 0$) and form the sheaf of rings:
\[
\sR(X, \sL, s) = \O_X \oplus \sL^{-1} \oplus \cdots \oplus \sL^{-n+1}
\]
where for $i,j < n$ and $i+j>n$, the multiplication $\sL^{-i} \otimes \sL^{-j} \to \sL^{-i-j+n}$ is performed by the formula $a \otimes b \mapsto a b s$. We define $\nu : Y = Y_{\sL, s} = \sheafspec \sR(X, \sL, s) \to X$.  Note we did not assume that $s$ was nonzero or even locally a nonzero divisor.
\end{setting}

Now let us work locally on an affine chart $U$ trivializing $\sL$, where $U= \Spec R \subseteq X$ with corresponding $\nu^{-1}(U) = \Spec S$.  Fix an isomorphism $\sL|_U \cong \O_U$ and write
\[
S = R \oplus Rt \oplus Rt^2 \oplus\cdots\oplus Rt^{n-1}
\]
where $t$ is a dummy variable used to help keep track of the degree.  The choice of the section $s|_U \in H^0(X, \sL^n)$ is the same as the choice of map $\O_X \to \sL^n$ (send $1$ to the section $s$).  If $s$ is chosen to be general and $\sL^n$ is very ample then this map is an inclusion, but it is not injective in general.  Working locally where we have fixed $\sL|_U = \O_U$, we also have implicitly chosen $\sL^n = \O_U$ and hence we have just specified that $t^n = s|_U \in \Gamma(U, \sL^n) = \Gamma(U, \O_U)$.  In other words
\[
S = R[t]/\langle t^n - s\rangle.
\]
In particular, it follows that $\nu$ is a finite map.

\begin{lemma}
\label{lem.FunctorialCyclicCover}
The map $\nu$ is functorial with respect to taking closed subschemes of $X$.  In particular if $Z \subseteq X$ is a closed subscheme and $\sL$ and $s$ are as above, then we have a commutative diagram
\[
\xymatrix{
Y_{\sL, s} \ar@{=}[r] & \sheafspec \sR(X, \sL, s) \ar[r]^-{\nu_X} & X \\
W_{\sL|_Z, s|_Z}\ar@{=}[r] & \sheafspec \sR(Z, \sL \otimes \O_Z, s|_Z) \ar[r]_-{\nu_Z} \ar@{^{(}->}[u] & Z. \ar@{^{(}->}[u]
}
\]
Furthermore, $W_{\sL|_Z, s|_Z} =  \pi^{-1}(Z)$ is the scheme theoretic preimage of $Z$.
\end{lemma}
\begin{proof}
The first statement follows since $\sR(X, \sL, s) = \O_X \oplus \sL^{-1} \oplus \ldots \oplus \sL^{-n+1}$ surjects onto $\sR(Z, \sL|_Z, s|_Z) =\O_Z \oplus \sL|_Z^{-1} \oplus \ldots \oplus \sL|_Z^{-n+1}$ in the obvious way.  For the second statement we simply notice that
\[
\sR(X, \sL, s) \otimes_{\O_X} \O_Z = \sR(Z, \sL|_Z, s|_Z).
\]
\end{proof}

\begin{lemma}
\label{lem.GeneralCyclicCoverOfReducedIsReduced}
Suppose we are working in \autoref{set.GeneralCyclicCover} and that $X$ is projective and also reduced (respectively normal), $\sL^n$ is globally generated and $s$ is chosen to be general.  Then $Y = Y_{\sL, s}$ is also reduced (respectively normal).
\end{lemma}
\begin{proof}
To show that $Y$ is reduced we will show it is \Sn{1} and \Rn{0}.  We work locally on an open affine chart $U = \Spec R \subseteq X$.

To show it is \Rn{0}, it suffices to consider the case $R = K$ is a field.  Because $s$ was picked generally, we may assume that the image of $s$ in $K$ is nonzero (we identify $s$ with its image).  Then we need to show that $K[t]/\langle t^n - s \rangle$ is a product of fields.  But this is obvious since we are working in characteristic zero.

Next we show it is \Sn{1}.  Indeed if $(R, \mathfrak{m})$ is a local ring of depth $\geq 1$ then obviously $S = R[t]/\langle t^n - s \rangle$ has depth $\geq 1$ as well since it is a free $R$-module.  Finally, the depth condition is preserved after localizing at the maximal ideals of the semi-local ring of $S$.

This proves the reduced case.  The normal case is well known but we sketch it.  The \Rn{1} condition follows from \cite[Lemma 2.51]{KollarMori} and the fact that $s$ is chosen generally, utilizing Bertini's theorem.  The \Sn{2} condition follows in the same way that \Sn{1} followed above.
\end{proof}

\subsection{$F$-pure and $F$-injective singularities} Du Bois singularities are conjectured to be the characteristic $0$ analog of $F$-injective singularities \cite{SchwedeFInjectiveAreDuBois}, \cite{BhattSchwedeTakagiweakordinaryconjectureandFsingularity}. In this short subsection we collect some definitions about singularities in positive characteristic. Our main focus are $F$-pure and $F$-injective singularities.

A local ring $(R,\m)$ is called {\it $F$-pure} if the Frobenius endomorphism $F$: $R\rightarrow R$ is pure.\footnote{A map of $R$-modules $N\rightarrow N'$ is pure if for every $R$-module $M$ the map $N\otimes_RM\rightarrow N'\otimes_RM$ is injective. This implies that $N\rightarrow N'$ is injective, and is weaker than the condition that $0\rightarrow N\rightarrow N'$ be split.} Under mild conditions, for example when $R$ is $F$-finite, which means the Frobenius map $R\xrightarrow{F}R$ is a finite map, $R$ being $F$-pure is equivalent to the condition that the Frobenius endomorphism $R\xrightarrow{F} R$ is split \cite[Corollary 5.3]{HochsterRobertsFrobeniusLocalCohomology}. The Frobenius endomorphism on $R$ induces a natural Frobenius action on each local cohomology module $H_\m^i(R)$ and we say a local ring is {\it $F$-injective} if this natural Frobenius action on $H_\m^i(R)$ is injective for every $i$ \cite{FedderFPureRat}. This holds if $R$ is $F$-pure \cite[Lemma 2.2]{HochsterRobertsFrobeniusLocalCohomology}. For some other basic properties of $F$-pure and $F$-injective singularities, see \cite{HochsterRobertsFrobeniusLocalCohomology,FedderFPureRat,EnescuHochsterTheFrobeniusStructureOfLocalCohomology}.

\section{The Cohen-Macaulay property for families of Du Bois pairs}

We begin with a lemma which we assume is well known to experts, indeed it is explicitly stated in \cite[Corollary 6.9]{KollarKovacsSingularitiesBook}.  However, because many of the standard references also include implicit reducedness hypotheses, we include a careful proof that deduces it from the reduced case.  Of course, one could also go back to first principals but we believe the path we take below is quicker.

\begin{lemma}
\label{surjectivity in the non reduced setting}
Let $X$ be a projective scheme over $\mathbb{C}$ and let $Z\subseteq X$ be a closed subscheme ($X, Z$ are not necessarily reduced). Then the natural map $$H^i(X, \scr{I}_Z)\to\mathbb{H}^i(X_{\red}, \underline{\Omega}_{X_{\red}, Z_{\red}}^0)\cong \mathbb{H}^i(X, \underline{\Omega}_{X, Z}^0)$$ is surjective for every $i\in\mathbb{Z}$.
\end{lemma}
\begin{proof}  Note that if the result holds for $\bC$, it certainly also holds for other fields of characteristic zero.
The isomorphism in the statement of the lemma follows from the definition.  For the surjectivity, we consider the following commutative diagram where we let $U=X\backslash Z$.  Note that the rows are not exact.
\[\xymatrix{
H^i_c(U_{\red}, \mathbb{C}) \ar[r] & H^i(X_{\red}, \scr{I}_{Z_{\red}}) \ar[r] & \mathbb{H}^i(X_{\red}, \underline{\Omega}_{X_{\red}, Z_{\red}}^0)\\
H^i_c(U, \mathbb{C}) \ar[r] \ar[u]^{\cong} & H^i(X, \scr{I}_Z) \ar[r] \ar[u] & \mathbb{H}^i(X, \underline{\Omega}_{X, Z}^0) \ar[u]^{\cong}
}\]
The composite map in the top horizontal line is a surjection by \cite[Lemma 2.17]{KovacsSchwedeInversionofadjunctionforDB} or \cite[Corollary 4.2]{KovacsDBPairsAndVanishing}.  The vertical isomorphism on the left holds because the constant sheaf $\mathbb{C}$ does not see the non-reduced structure.  Likewise for the vertical isomorphism on the right where $\DuBois{X,Z}$ does not see the non-reduced structure.  The diagram then shows that $H^i(X, \scr{I}_Z)\to\mathbb{H}^i(X_{\red}, \underline{\Omega}_{X_{\red}, Z_{\red}}^0)$ is a surjection.
\end{proof}

Next we prove the key injectivity lemma for possibly non-reduced pairs. The proof is essentially the same as in \cite{KovacsSchwedeInversionofadjunctionforDB} or \cite{KovacsSchwedeDBDeforms}. We reproduce it here carefully because we need it in the non-reduced setting.

\begin{lemma}
\label{key injectivity}
Let $X$ be a scheme of essentially finite type over $\mathbb{C}$ and $Z\subseteq X$ a closed subscheme ($X$ and $Z$ are not necessarily reduced). Then the natural map
$$h^j(\underline{\omega}^\mydot_{X,Z})\hookrightarrow h^j(\myR\sheafhom_{\O_X}(\scr{I}_Z, \omega^\mydot_X))$$
is injective for every $j\in\mathbb{Z}$, where $\underline{\omega}^\mydot_{X,Z}=\myR\sheafhom_{\O_X}(\underline{\Omega}_{X, Z}^0, \omega_X^\mydot)$.
\end{lemma}
\begin{proof}
The question is local and compatible with restricting to an open subset, hence we may assume that $X$ is projective with ample line bundle $\scr{L}$. Let $s\in H^0(X, \scr{L}^n)$ be a general global section for some $n \gg 0$ and let $\eta$: $Y=\mathbf{Spec}\oplus_{i=0}^{n-1}\scr{L}^{-i}\to X$ be the $n$-th cyclic cover corresponding to $s$. Set $W=\eta^{-1}(Z)$. Then for $n\gg 0$ and $s$ general, the restriction of $\eta$ to $W$ is the cyclic cover $W=\mathbf{Spec}\oplus_{i=0}^{n-1}\scr{L}|_Z^{-i}\to Z$ by \autoref{lem.FunctorialCyclicCover}.  Likewise $\eta$ also induces the corresponding cyclic covers of the closed subschemes $Y_{\red}\to X_{\red}$ and $W_{\red}\to Z_{\red}$ by \autoref{lem.FunctorialCyclicCover} and \autoref{lem.GeneralCyclicCoverOfReducedIsReduced}. We have $\eta_*\scr{I}_W=\oplus_{i=0}^{n-1}\scr{I}_Z\otimes\scr{L}^{-i}$ and
\[
\eta_*\underline{\Omega}_{Y, W}^0\cong\eta_*\underline{\Omega}_{Y_{\red}, W_{\red}}^0\cong\oplus_{i=0}^{n-1}{\underline{\Omega}}_{X_{\red}, Z_{\red}}^0\otimes\scr{L}^{-i}\cong\oplus_{i=0}^{n-1}{\underline{\Omega}}_{X, Z}^0\otimes\scr{L}^{-i}
\]
where the second isomorphism is \cite[Lemma 3.1]{KovacsSchwedeInversionofadjunctionforDB} (see also \cite[Lemma 3.1]{KovacsSchwedeDBDeforms}). Since \autoref{surjectivity in the non reduced setting} implies that $H^j(Y, \scr{I}_W)\twoheadrightarrow \bH^j(Y, \underline{\Omega}_{Y, W}^0)$ is surjective for every $j\in\mathbb{Z}$, we know that $H^j(X, \scr{I}_Z\otimes\scr{L}^{-i})\twoheadrightarrow \bH^j(X, \underline{\Omega}_{X, Z}^0\otimes\scr{L}^{-i})$ is surjective for every $i\geq0$ and $j\in\mathbb{Z}$.

Applying Grothendieck-Serre duality we obtain an injection
$$\mathbb{H}^j(X, \underline{\omega}_{X,Z}^\mydot\otimes\scr{L}^i)\hookrightarrow\mathbb{H}^j(X, \myR\sheafhom_{\O_X}(\scr{I}_Z, \omega_X^\mydot)\otimes\scr{L}^i)$$ for all $i\geq 0$ and $j\in\mathbb{Z}$. Since $\scr{L}$ is ample, for $i\gg0$ the spectral sequence computing the above hypercohomology degenerates. Hence for $i\gg0$ we get
$$H^0(X, h^j(\underline{\omega}_{X,Z}^\mydot)\otimes\scr{L}^i)\hookrightarrow H^0(X, h^j(\myR\sheafhom_{\O_X}(\scr{I}_Z, \omega_X^\mydot))\otimes\scr{L}^i).$$ But again since $\scr{L}$ is ample, the above injection for $i \gg 0$ implies the injection $h^j(\underline{\omega}_{X,Z})\hookrightarrow h^j(\myR\sheafhom_{\O_X}(\scr{I}_Z, \omega_X^\mydot))$.
\end{proof}

Next we prove the key lemma stated in the introduction (and we do the generalized pair version). It follows immediately from the above injectivity, \autoref{key injectivity}. For simplicity we will use $(S,S/J)$ to denote the pair $(\Spec S, \Spec S/J)$.

\begin{lemma}
\label{surjectivity in local cohomology}
Let $(S,\m)$ be a local ring essentially of finite type over $\mathbb{C}$ and let $J\subseteq S$ be an ideal. Suppose further that $S' = S/N$ where $N \subseteq S$ is an ideal contained in the nilradical (for instance, $N$ could be the nilradical and then $S' = S_{\red}$).  Suppose $(S', S'/JS')$ is a Du Bois pair. Then the natural map $H_\m^i(J)\to H_\m^i(JS')$ is surjective for every $i$.  In particular, if $S_{\red}$ is Du Bois, then $H_\m^i(S)\to H_\m^i(S_{\red})$ is surjective for every $i$.
\end{lemma}
\begin{proof} We consider the following commutative diagram:
\[
\xymatrix{
H_\m^i(J) \ar[r] \ar@{->>}[d] & H^i_\m(JS') \ar[d]^\cong\\
\mathbb{H}_\m^i(\underline{\Omega}^0_{S, S/J}) \ar[r]^-\cong & \mathbb{H}^i_\m(\underline{\Omega}^0_{S', S'/JS'})
}
\]
Here the left vertical map is surjective by the Matlis dual of \autoref{key injectivity} applied to $X=\Spec S$ and $Z=\Spec S/J$, the right vertical map is an isomorphism because $(S', S'/JS')$ is a Du Bois pair, and the bottom map is an isomorphism because $S_{\red}=S'_{\red}$. Chasing the diagram shows that $H_\m^i(J)\to H_\m^i(JS')$ is surjective. The last sentence is the case that $J$ is the unit ideal (i.e., the non-pair version).
\end{proof}

\begin{remark}
\label{rem.SmallDefOfLocalCohomologyAndFPure}
The characteristic $p>0$ analog of the above lemma (in the case $J=S$, i.e., the non-pair version) holds if $S_{\red}$ is $F$-pure. We give a short argument here. Without loss of generality we may assume $S$ and $S_{\red}$ are complete, so $S=A/I$ and $S_{\red}=A/\sqrt{I}$ where $A$ is a complete regular local ring. We may pick $e\gg0$ such that $(\sqrt{I})^{[p^e]}\subseteq I$. We have a composite map $$H_\m^i(A/(\sqrt{I})^{[p^e]})\to H_\m^i(A/I)\cong H_\m^i(S)\xrightarrow{\phi} H_\m^i(A/\sqrt{I})\cong H_\m^i(S_{\red}).$$ We know from \cite[Lemma 2.2]{LyubeznikVanishingoflocalcohomologycharP} that the image of the composite map is equal to $\text{span}_A\langle f^e(H_\m^i(S_{\red}))\rangle$, where $f^e$ denotes the natural $e$-th Frobenius action on $H_\m^i(S_{\red})$. In particular, the Frobenius action on $H_\m^i(S_{\red})/\im\phi$ is nilpotent. However, since $\im\phi$ is an $F$-stable submodule of $H_\m^i(S_{\red})$ and $S_{\red}$ is $F$-pure, \cite[Theorem 3.7]{MaFinitenesspropertyoflocalcohomologyforFpurerings} shows that Frobenius acts injectively on $H_\m^i(S_{\red})/\im\phi$. Hence we must have $H_\m^i(S_{\red})=\im\phi$, that is, $H_\m^i(S)\to H_\m^i(S_{\red})$ is surjective.
\end{remark}

We next give an example showing that the analog of \autoref{surjectivity in local cohomology} for $F$-injectivity fails in general. The example is a modification of \cite[Example 2.16]{EnescuHochsterTheFrobeniusStructureOfLocalCohomology}.

\begin{example}
\label{example:nonsurjectivityforF-inj}
Let $K=k(u,v)$ where $k$ is an algebraically closed field of characteristic $p>0$ and let $L=K[z]/(z^{2p}+uz^p+v)$ as in \cite[Example 2.16]{EnescuHochsterTheFrobeniusStructureOfLocalCohomology}. Now let $R$=$K+(x,y)L[[x,y]]$ with $\m=(x,y)L[[x,y]]$. Then it is easy to see that $(R,\m)$ is a local ring of dimension $2$ and we have a short exact sequence: $$0\to R\to L[[x,y]]\to L/K\to 0.$$ The long exact sequence of local cohomology gives $$H_\m^1(R)\cong L/K, \hspace{1em} H_\m^2(R)\cong H_\m^2(L[[x,y]]).$$ Moreover, one can check that the Frobenius action on $H_\m^1(R)$ is exactly the Frobenius action on $L/K$. The Frobenius action on $L/K$ is injective since $L^p\cap K=K^p$, and the Frobenius action on $H_\m^2(L[[x,y]])$ is injective because $L[[x,y]]$ is regular. Hence the Frobenius action on both $H_\m^1(R)$ and $H_\m^2(R)$ are injective. This proves that $R$ is $F$-injective.

Write $R=A/I$ for a regular local ring $(A,\m)$. One checks that the Frobenius action $F$: $H_\m^1(R)\to H_\m^1(R)$ is not surjective up to $K$-span (and hence not surjective up to $R$-span since the residue field of $R$ is $K$) because $L\neq L^p[K]$ by our choice of $K$ and $L$ (see \cite[Example 2.16]{EnescuHochsterTheFrobeniusStructureOfLocalCohomology} for a detailed computation on this). But now \cite[Lemma 2.2]{LyubeznikVanishingoflocalcohomologycharP} shows that $H_\m^1(A/I^{[p]})\to H_\m^1(A/I)$ is not surjective. Therefore we can take $S=A/I^{[p]}$ with $S_{\red}=R$ that is $F$-injective, but $H_\m^1(S)\to H_\m^1(S_{\red})$ is not surjective.
\end{example}

In view of \autoref{rem.SmallDefOfLocalCohomologyAndFPure} and \autoref{example:nonsurjectivityforF-inj} and the relation between $F$-injective and Du~Bois singularities, it is tempting to try to define a more restrictive variant of $F$-injective singularities for local rings.  In particular, if $(R, \fram)$ is a local ring such that $R_{\red}$ has these more restrictive $F$-injective singularities, then it should follow that $H^i_{\fram}(R) \to H^i_{\fram}(R_{\red})$ surjects for all $i$.


\begin{theorem}
\label{theorem.SheafversionSurj}
Suppose that $(X, Z)$ is a reduced pair and that $H \subseteq X$ is a Cartier divisor such that $H$ does not contain any irreducible components of either $X$ or $Z$.  If $(H, Z \cap H)$  is a Du Bois pair, then for all $i$ and all points $\eta \in H \subseteq X$, the following segment of the long exact sequence
\[
0 \to H^i_\eta(\sI_{Z, \eta} \cdot \O_{H, \eta}) \hookrightarrow H^{i+1}_\eta(\sI_{Z, \eta} \cdot \O_{X, \eta}(-H)) \twoheadrightarrow H^{i+1}_\eta(\sI_{Z, \eta}) \to 0
\]
is exact for all $i$.  Dually, in the special case that $\sI_Z = \O_X$, we can also phrase this as saying that
\[
0 \to \myH^{-i}( \omega_X^{\mydot}) \to \myH^{-i}( \omega_X^{\mydot}(H)) \to \myH^{-i+1}( \omega_H^{\mydot}) \to 0
\]
is exact for all $i$.
\end{theorem}
\begin{proof}
Localizing at $\eta$, we may assume that $X=\Spec R$, $Z=\Spec R/I$, $H=\Div(x)$ with $(R,\m)$ a local ring and $\eta=\{\m\}$. Moreover, the hypotheses imply that $x$ is a nonzerodivisor on both $R$ and $R/I$. It is enough to show that the segment of the long exact sequence
\begin{equation}
\label{equation LES local coho}
0\to H_\m^i(I/xI)\to H_\m^{i+1}(I)\xrightarrow{\cdot x} H_\m^{i+1}(I)\to 0
\end{equation}
induced by $0\to I\xrightarrow{\cdot x}I \to I/xI\to 0$ is exact.

Consider $S=R/x^nR$ and $J=I(R/x^nR)$.  The pair $(S' = R/xR, S'/IS')=(H, Z\cap H)$ is Du Bois by hypothesis (note we do not assume that $H$ is reduced).  Thus \autoref{surjectivity in local cohomology} implies that
\begin{equation}
\label{equation 1}
H_\m^i(I(R/x^nR))\twoheadrightarrow H_\m^i(I(R/xR))
\end{equation}
is surjective for every $i$, $n$. Since $x$ is a nonzerodivisor on $R/I$, $x^nI=I\cap x^nR$. Thus
\[
I(R/x^nR)=\frac{I+x^nR}{x^nR}\cong\frac{I}{I\cap x^nR}=I/x^nI.
\]
 Hence (\ref{equation 1}) shows that
$$H_\m^i(I/x^nI)\twoheadrightarrow H_\m^i(I/xI)$$ is surjective for every $i$, $n$. The long exact sequence of local cohomology induced by $0\to I/x^{n-1}I\xrightarrow{\cdot x} I/x^nI\to I/xI\to 0$ tells us that $$H_\m^i(I/x^{n-1}I)\xrightarrow{\cdot x} H_\m^i(I/x^nI)$$ is injective for every $i$, $n$. Hence after taking a direct limit we obtain that:
\[
\phi_i: H_\m^i(I/xI)\to \varinjlim_n H_\m^i(I/x^{n}I)\cong H_\m^i(\varinjlim_n (I/x^{n}I))\cong H_\m^i(H_x^1(I))\cong H_\m^{i+1}(I)
 \]
is injective for every $i$. Here the last two isomorphism come from the fact that $x$ is a nonzerodivisor on $I$ (since it is a nonzerodivisor on $R$) and a simple computation using the local cohomology spectral sequence.
\begin{claim}
The $\phi_i$'s are exactly the connection maps in the long exact sequence of local cohomology induced by $0\to I\xrightarrow{\cdot x} I\to I/xI\to 0$:
\[
\cdots\to H_\m^i(I) \to H_\m^i(I/x I)\xrightarrow{\phi_i} H_\m^{i+1}(I) \xrightarrow{\cdot x} H_\m^{i+1}(I)\to\cdots.
 \]
\end{claim}

This claim immediately produces (\ref{equation LES local coho}) because we proved that $\phi_i$ is injective for every $i$. Thus it remains to prove the Claim.
\begin{proof}[Proof of claim]
Observe that by definition, $\phi_i$ is the natural map in the long exact sequence of local cohomology
$$\cdots\to H_\m^i(I/xI)\xrightarrow{\phi_i} H_\m^i(I_x/I) \xrightarrow{\cdot x} H_\m^i(I_x/I)\to\cdots$$
which is induced by $0\to I/xI\to I_x/I \xrightarrow{\cdot x} I_x/I \to 0$ (note that $x$ is a nonzerodivisor on $I$ and $H_x^1(I)\cong I_x/I$). However, it is easy to see that the multiplication by $x$ map $H_\m^i(I_x/I) \xrightarrow{\cdot x} H_\m^i(I_x/I)$ can be identified with the multiplication by $x$ map $H_\m^{i+1}(I)\xrightarrow{\cdot x} H_\m^{i+1}(I)$ because we have a natural identification $H_\m^i(I_x/I)\cong H_\m^i(H_x^1(I))\cong H_\m^{i+1}(I)$. This finishes the proof of the claim and thus also the local cohomology statement.
\end{proof}

The global statement can be checked locally, where it is simply a special case of the local dual of the local cohomology statement.
\end{proof}

The following is the main theorem of the section, it answers a question of Kov\'{a}cs and the second author \cite[Question 5.7]{KovacsSchwedeInversionofadjunctionforDB} and is a generalization to pairs of a characteristic zero analog of \cite[Corollary A.5]{HoriuchiMillerShimomoto}.

\begin{theorem}
\label{inversion of adjunction for DB pair}
Suppose that $(X, Z)$ is a reduced pair and that $H\subseteq X$ is a Cartier divisor such that $H$ does not contain any irreducible components of either $X$ or $Z$. If $(H, Z\cap H)$ is a Du Bois pair and $\scr{I}_Z|_{X\backslash H}$ is \Sn{k}, then $\scr{I}_Z$ is \Sn{k}.
\end{theorem}
\begin{proof}
Suppose $\scr{I}_Z$ is not \Sn{k}, choose an irreducible component $Y$ of the non-\Sn{k} locus of $\scr{I}_Z$. Since $\scr{I}_Z|_{X\backslash H}$ is \Sn{k}, we know that $Y\subseteq H$. Let $\eta$ be the generic point of $Y$. \autoref{theorem.SheafversionSurj} tells us that $$H^{i}_\eta(\sI_Z \cdot \O_X(-H)) \twoheadrightarrow H^{i}_\eta(\sI_Z)$$ is surjective for every $i$. Note that if we localize at $\eta$ and use the same notation as in the proof of \autoref{theorem.SheafversionSurj}, this surjection is the multiplication by $x$ map $H_\m^i(I) \xrightarrow{\cdot x} H_\m^i(I)$.

However, after we localize at $\eta$, $I$ is \Sn{k} on the punctured spectrum, $H_\m^i(I)$ has finite length for every $i<k$. In particular, for $h\gg 0$, $x^h$ annihilates $H_\m^i(I)$ for every $i<k$. Therefore for $i<k$, the multiplication by $x$ map $H_\m^i(I) \xrightarrow{\cdot x} H_\m^i(I)$ cannot be surjective unless $H_\m^i(I)=0$, which means $I$ is \Sn{k} as an $R$-module. But this contradicts our choice of $Y$ (because we pick $Y$ an irreducible component of the non-\Sn{k} locus of $\scr{I}_Z$). Therefore $\scr{I}_Z$ is \Sn{k}.
\end{proof}

\begin{corollary}
\label{cor.CMNearCartier}
Let $X$ be a reduced scheme and let $H\subseteq X$ be a Cartier divisor. If $H$ is Du Bois and $X\backslash H$ is Cohen-Macaulay, then $X$ is Cohen-Macaulay and hence so is $H$.
\end{corollary}

As mentioned in the introduction, these results should be viewed as a generalization of \cite[Corollary 1.3, Corollary 7.13]{KollarKovacsLCImpliesDB} and \cite[Theorem 5.5]{KovacsSchwedeInversionofadjunctionforDB}.  In those results, one considered a flat projective family $X \to B$ with Du~Bois fibers such that the general fiber is \Sn{k}.  They then show that the special fiber is $S_k$.  Let us explain this result in a simple case. Suppose that $B$ is smooth curve and $Z=0$, the special fiber is a Cartier divisor $H$ and $X\setminus H$ is Cohen-Macaulay.  Then we are trying to show that $H$ is Cohen-Macaulay which is of course the same as showing that $X$ is Cohen-Macaulay, which is exactly what our result proves.  The point is that we had to make no projectivity hypothesis for our family.

This result also implies \cite[Conjecture 7.9]{KovacsSchwedeInversionofadjunctionforDB} which we state next.  We leave out the definition of a rational pair here (this is a rational pair in the sense of Koll\'ar and Kov\'acs: a generalization of rational singularities analogous to dlt singularities, see \cite{KollarKovacsSingularitiesBook}), the interested reader can look it up, or simply assume that $D = 0$.

\begin{corollary}
\label{cor.RatOutsidePlusDBinImpliesRat}
Suppose $(X, D)$ is a pair with $D$ a reduced Weil divisor.  Further suppose that $H$ is a Cartier divisor on $X$ not having any components in common with $D$ such that $(H, D \cap H)$ is Du~Bois and such that $(X \setminus H, D \setminus H)$ is a rational pair in the sense of Koll\'ar and Kov\'acs.  Then $(X, D)$ is a rational pair in the sense of Koll\'ar and Kov\'acs.
\end{corollary}
\begin{proof}
We follow the \cite[Proof of Theorem 7.1]{KovacsSchwedeInversionofadjunctionforDB}.  It goes through word for word once one observes that $\O_X(-D)$ is Cohen-Macaulay which follows immediately from \autoref{inversion of adjunction for DB pair}.
\end{proof}

This corollary is also an analog of \cite[Proposition 2.13]{FedderWatanabe} and \cite[Corollary A.4]{HoriuchiMillerShimomoto}.

\section{Applications to local cohomology}

In this section we give many applications of \autoref{surjectivity in local cohomology} to local cohomology, several of these applications are quite easy to prove. However, they provide strong and surprising characteristic $0$ analogs and generalizations of results in \cite{HochsterRobertsFrobeniusLocalCohomology,SinghandWaltherArithmeticrankofSegreproducts,SinghandWaltherlocalcohomologyandpure,VarbaroCohomologicaldimension,DeStefaniNunezBetancourtFthresholdGradedrings} as well as answer their questions. Moreover, our results can give a generalization of the classical Kodaira vanishing theorem.

\subsection{Injectivity of Ext and a generalization of the Kodaira vanishing theorem} Our first application gives a solution to \cite[Question 6.2]{EisenbudMustataStillmanCohomologyontoricvarieties} on the injectivity of $\Ext$ in characteristic $0$, which parallels to its characteristic $p>0$ analog \cite[Theorem 1.3]{SinghandWaltherlocalcohomologyandpure}.
\begin{proposition}
\label{injectivity of Ext}
Let $(R,\m)$ be a Gorenstein local ring essentially of finite type over $\mathbb{C}$, and let $I\subseteq R$ be an ideal such that $R/I$ has Du~Bois singularities. Then the natural map $\Ext^j_R(R/I, R)\to H_I^j(R)$ is injective for every $j$.
\end{proposition}
\begin{proof}
By \autoref{surjectivity in local cohomology} applied to $S=R/I^t$ and applying local duality, we know that $$\Ext^j_R(R/I, R)\to \Ext^j_R(R/I^t, R)$$ is injective for every $j$ and every $t>0$. Now taking a direct limit we obtain the desired injectivity \mbox{$\Ext^j_R(R/I, R)\to H_I^j(R)$}.
\end{proof}

\begin{remark}
\cite[Theorem 1.3]{SinghandWaltherlocalcohomologyandpure} proves the same injectivity result in characteristic $p>0$ when $R/I$ is $F$-pure (and when $R$ is regular). Perhaps it is worth to pointing out that \autoref{injectivity of Ext} fails in general if we replace Du Bois by $F$-injective: \autoref{example:nonsurjectivityforF-inj} is a counterexample. Because there we have $S=A/I$ is $F$-injective but $H_\m^1(A/I^{[p]})\to H_\m^1(A/I)$ is not surjective. Hence applying local duality shows that $$\Ext^{\dim A-1}_A(A/I, A)\to \Ext^{\dim A-1}_A(A/I^{[p]}, A)$$ is not injective, so neither is $\Ext^{\dim A-1}_A(A/I, A)\to H_I^{\dim A-1}(A)$.
\end{remark}

An immediate corollary of \autoref{injectivity of Ext} is the following, which is a characteristic $0$ analog of \cite[Theorem 7.3]{DeStefaniNunezBetancourtFthresholdGradedrings} (this also answers \cite[Question 7.5]{DeStefaniNunezBetancourtFthresholdGradedrings}, which is motivated by Stillman's conjecture).


\begin{corollary}
Let $R$ be a regular ring of essentially finite type over $\mathbb{C}$. Let $I\subseteq R$ be an ideal such that $S=R/I$ has Du Bois singularities. Then $\pd_RS\leq\nu(I)$ where $\nu(I)$ denotes the minimal number of generators of $I$.
\end{corollary}
\begin{proof}
For every $j>\nu(I)$ and every maximal ideal $\m$ of $R$, we have $H_{IR_\m}^j(R_\m)=0$. Therefore by \autoref{injectivity of Ext}, we have $\Ext^j_{R_\m}(S_\m, R_\m)=0$ for every $j>\nu(I)$. Since $\pd_{R_\m}S_\m<\infty$, we have $$\pd_{R_\m}S_\m=\sup\{j|\Ext^j_{R_\m}(S_\m, R_\m)\neq0\}\leq \nu(I)\footnote{This follows from applying the functor $\Hom_{R_{\m}}(-, R_{\m})$ to a minimal free resolution of $S_{\m}$ and observing that the matrix defining the maps have elements contained in $\m$.  See also \cite[2.4]{DwyerGreenleesIyengarFinitenessInDerivedCategories}}.$$ Because this is true for every maximal ideal $\m$, we have $\pd_RS\leq\nu(I)$.
\end{proof}

Next we want to prove a stronger form of \autoref{injectivity of Ext} in the graded case. We will first need the following criterion of Du Bois singularities for $\mathbb{N}$-graded rings. This result and related results should be well-known to experts (or at least well-believed by experts): for example \cite[Theorem 4.4]{MaF-injectivityandBuchsbaumsingularities}, \cite[Lemma 2.14]{BhattSchwedeTakagiweakordinaryconjectureandFsingularity}, \cite[Lemma 2.12]{KovacsSchwedeInversionofadjunctionforDB} or \cite{GK13}. However all these references only deal with the case that $R$ is the section ring of a projective variety with respect to a certain ample line bundle, while here we allow arbitrary $\mathbb{N}$-graded rings which are not even normal. We could not find a reference that handles the generality that we will need, and hence we write down a careful argument.

\begin{proposition}
\label{prop.CriterionforDB}
Let $(R,\m)$ be a reduced Noetherian $\mathbb{N}$-graded $(R_0=k)$-algebra with $\m$ the unique homogeneous maximal ideal. Suppose $R_P$ is Du Bois for all $P\neq\m$. Then we have $$h^i(\DuBois{R})\cong \big[H_\m^{i+1}(R)\big]_{>0}$$ for every $i>0$ and $$h^0(\DuBois{R})/R\cong \big[H^1_\m(R)\big]_{>0}.$$ In particular, $R$ is Du Bois if and only if $\big[H_\m^i(R)\big]_{>0}=0$ for every $i>0$.
\end{proposition}
\begin{proof}
Let $R^\natural$ denote the Rees algebra of $R$ with respect to the natural filtration $R_{\geq t}$. That is, $R^\natural=R\oplus R_{\geq1}\oplus R_{\geq2}\oplus\cdots.$  Let $Y=\Proj R^\natural$. We first claim that $Y$ is Du Bois: $Y$ is covered by $D_+(f)$ for homogeneous elements $f\in R_{\geq t}$ and $t\in\mathbb{N}$.

If $\deg f>t$, then $\big[R^\natural_{f}\big]_0\cong R_f$ is Du Bois.  If $\deg f=t$, then $\big[R^\natural_{f}\big]_0\cong \big[R_f\big]_{\geq0}$ is also Du Bois\footnote{In general, if $S$ is a $\mathbb{Z}$-graded ring, then $S_{\geq0}\cong S[z]_0$ where $\deg z=-1$. Hence if $S$ is Du Bois, then $S_{\geq0}$, being a summand of $S[z]$, is also Du Bois \cite{KovacsDuBoisLC1}.} (see \cite[Lemma 5.4]{KuranoSatoSinghWatanabeMultigradedrings} for an analogous analysis on rational singularities).

Since $R$ and thus $R^\natural$ are Noetherian, there exists $n$ such that $R_{\geq nt}=(R_{\geq n})^t$ for every $t\geq 1$ \cite[Chapter III, Section 3, Proposition 3]{Bourbaki1998}. Let $I=(R_{\geq n})$, then we immediately see that \mbox{$Y\cong \Proj\big(R\oplus I\oplus I^2\oplus\dots\big)$} is the blow up of $\Spec R$ at $I$. We define the exceptional divisor to be $E\cong \Proj\big(R/I\oplus I/I^2\oplus I^2/I^3\oplus\cdots\big)$. We next claim that
\begin{equation}
\label{equation reduced exceptional}
(R/I\oplus I/I^2\oplus I^2/I^3\oplus\cdots)_{\red}\cong R_0\oplus R_n\oplus R_{2n}\oplus\cdots.
\end{equation}
The point is that, for $\overline{x}\in I^t/I^{t+1}=R_{\geq nt}/R_{\geq n(t+1)}$, if $x\in R_{nt+a}$ for some $a>0$, then we can pick $b>0$ such that $ba>n$, we then have $$x^b\in R_{bnt+ba}\subseteq R_{\geq n(bt+1)}.$$ But this means $\overline{x}^b=0$ in $I^{bt}/I^{bt+1}$ and thus $\overline{x}$ is nilpotent in $R/I\oplus I/I^2\oplus I^2/I^3\oplus\cdots$. This proves (\ref{equation reduced exceptional}).

By (\ref{equation reduced exceptional}) we have $E_{\red}\cong \Proj R_0\oplus R_n\oplus R_{2n}\oplus\cdots\cong\Proj R$ is Du Bois (because $\big[R_f\big]_0$ is Du Bois for every homogeneous $f\in R$). We consider the following commutative diagram:
\[\xymatrix{
E \ar[r]\ar[d]^\pi & Y\ar[d]^\pi\\
\Spec R/I \ar[r] & \Spec R
}\]
Since $Y$, $E_{\red}$ and $(\Spec R/I)_{\red}\cong \Spec k$ are all Du Bois by the above discussion, the exact triangle $\DuBois{R}\to\myR\pi_*\DuBois{Y}\oplus \DuBois{R/I}\to \myR\pi_*\DuBois{E}\xrightarrow{+1}$ reduces to
\begin{equation}
\label{equation 2}
\DuBois{R}\to\myR\pi_*\O_Y\oplus k\to\myR\pi_*\O_{E_{\red}}\xrightarrow{+1}.
\end{equation}

Next we study the map $\myR\pi_*\O_Y\to \myR\pi_*\O_{E_{\red}}$ using the \v{C}ech complex. We pick $x_1,\dots,x_m\in I=(R_{\geq n})$ in $R\oplus I\oplus I^2\oplus\cdots$ such that
\begin{enumerate}
\item The internal degree, $\deg x_i=n$ for each $x_i$ (in other words, $x_i \in R_n \subseteq R_{\geq n} = I$);
\item The images $\overline{x}_1,\dots,\overline{x}_m$ in $(R/I\oplus I/I^2\oplus I^2/I^3\oplus\cdots)_{\red}=R_0\oplus R_n\oplus R_{2n}\oplus\cdots$ are algebra generators of $R_0\oplus R_n\oplus R_{2n}\oplus\cdots$ over $R_0=k$.
\end{enumerate}
Note that conditions (a) and (b) together imply that the radical of $(x_1,\dots,x_m)$ in $R\oplus I\oplus I^2\oplus\cdots$ is the irrelevant ideal $I\oplus I^2\oplus\cdots$. In particular, $\{D_+(x_i)\}_{1\leq i\leq m}$ forms an affine open cover of $Y$. The point is that for any $y\in I^{t}=R_{\geq tn}$, $y^n\in I^{tn}=R_{\geq tn^2}$ as an element in $R\oplus I\oplus I^2\oplus\cdots$ is always contained in the ideal $(x_1,\dots,x_m)$, this is because the internal degree of $y^n$ is divisible by $n$, so it can be written as a sum of monomials in $x_i$ by (b).

The natural map $\O_{Y}\rightarrow \O_{E_{\red}}$ induces a map between the $s$-th spot of the \v{C}ech complexes of $\O_{Y}$ and $\O_{E_{\red}}$ with respect to the affine cover $\{D_+(x_i)\}_{1\leq i\leq m}$. The induced map on \v{C}ech complexes can be explicitly described as follows (all the direct sums in the following diagram are taken over all $s$-tuples $1\leq i_1<\cdots <i_s\leq m$):
\[  \xymatrix{
\bigoplus \O_{Y}(D_+(x_{i_1}x_{i_2}\cdots x_{i_s}))  \ar[r] \ar[d]^{\cong} & \bigoplus \O_{E_{\red}}(D_+(\overline{x}_{i_1}\overline{x}_{i_2}\cdots \overline{x}_{i_s}))\ar[d]^{\cong}  \\
\bigoplus\left\{\frac{y}{(x_{i_1}x_{i_2}\cdots x_{i_s})^h}\Bigg| h> 0, y\in I^{sh}=R_{\geq nsh}\right\} \ar[r] \ar[d]^{\cong} & \bigoplus\left\{\frac{\overline{y}}{(\overline{x}_{i_1}\overline{x}_{i_2}\cdots \overline{x}_{i_s})^h}\Bigg| h> 0, y\in R_{nsh}\right\} \ar[d]^\cong \\
\bigoplus\big[R_{x_{i_1}x_{i_2}\cdots x_{i_s}}\big]_{\geq 0} \ar[r]^\phi & \bigoplus\big[R_{x_{i_1}x_{i_2}\cdots x_{i_s}}\big]_0   
} \]

The induced map on the second line takes the element $\D\frac{y}{(x_{i_1}x_{i_2}\cdots x_{i_s})^n}$ to $\D\frac{\overline{y}}{(\overline{x}_{i_1}\overline{x}_{i_2}\cdots \overline{x}_{i_s})^n}$ where $\overline{y}$ denotes the image of $y$ in $R_{nsh}$. Hence the same map $\phi$ on the third line is exactly ``taking the degree $0$ part". Therefore we have
\[
\myR^i\pi_*\O_Y\cong H^i(Y, \O_Y)=\big[H_\m^{i+1}(R)\big]_{\geq0}
 \]
while
\[
\myR^i\pi_*\O_{E_{\red}}\cong H^i(E_{\red}, \O_{E_{\red}})\cong \big[H_\m^{i+1}(R)\big]_0
 \]
 for every $i\geq 1$, and the map $\myR^i\pi_*\O_Y\to \myR^i\pi_*\O_{E_{\red}}$ is taking degree $0$ part. Therefore taking cohomology of (\ref{equation 2}), we have
\begin{equation}
\label{equation h^i for i>=2}
h^i(\DuBois{R})\cong \big[H_\m^{i+1}(R)\big]_{>0} \text{ for every } i\geq 2.
\end{equation}
Moreover, for $i=0,1$, the cohomology of (\ref{equation 2}) gives
\begin{equation}
\label{equation 3}
0\to h^0(\DuBois{R})\to H^0(Y, \O_Y)\oplus k\xrightarrow{\phi} H^0(E_{\red}, \O_{E_{\red}})\to h^1(\DuBois{R})\rightarrow \big[H_\m^2(R)\big]_{>0}\to 0.
\end{equation}
A similar \v{C}ech complex computation as above shows that
\[
H^0(Y, \O_Y)=\ker\Big(\oplus\big[R_{x_i}\big]_{\geq0}\to\oplus\big[R_{x_ix_j}\big]_{\geq0}\Big)
\]
while $H^0(E_{\red}, \O_{E_{\red}})=\ker\Big(\oplus\big[R_{x_i}\big]_{0}\to\oplus\big[R_{x_ix_j}\big]_{0}\Big)$. Therefore $\phi$ is surjective, which implies
\begin{equation}
\label{equation h^1}
h^1(\DuBois{R})\cong \big[H_\m^{2}(R)\big]_{>0}.
\end{equation}
Taking degree $>0$ part of (\ref{equation 3}) we get an exact sequence
\[
0\to h^0(\DuBois{R})_{>0}\to \oplus\big[R_{x_i}\big]_{>0}\to\oplus\big[R_{x_ix_j}\big]_{>0}.
\]
This implies that $h^0(\DuBois{R})_{>0}\cong (\Gamma(\Spec R \setminus \fram, \O_{\Spec R}))_{>0}$ and so
\begin{equation}
\label{equation h^0 in positive degree}
h^0(\DuBois{R})_{>0}/R_{>0}\cong \big[H_\m^1(R)\big]_{>0}.
\end{equation}

Finally we notice that $h^0(\DuBois{R})\cong R^{\sn}$ is the seminormalization of $R$ \cite[5.2]{SaitoMixedHodge}. We know $R^{\sn}_0\subseteq R^{\sn}$ is reduced. But we can also view $R^{\sn}_0$ as the quotient $R^{\sn}/R^{\sn}_{>0}$, in particular the prime ideals of $R^{\sn}_0$ correspond to prime ideals of $R^{\sn}$ that contain $R^{\sn}_{>0}$, so they all contract to $\m$ in $R$. Since seminormlization induces a bijection on spectrum, $R^{\sn}_0$ has a unique prime ideal. Thus $R^{\sn}_0$, being a reduced Artinian local ring, must be a field. Since seminormalization also induces isomorphism on residue fields, $R^{\sn}_0=k$ and thus $h^0(\DuBois{R})_0\cong R^{\sn}_0=R_0$. Hence (\ref{equation h^0 in positive degree}) tells us that
\begin{equation}
\label{equation h^0}
h^0(\DuBois{R})/R\cong \big[H_\m^1(R)\big]_{>0}.
\end{equation}
Now (\ref{equation h^i for i>=2}), (\ref{equation h^1}) and (\ref{equation h^0}) together finish the proof.
\end{proof}

Now we prove our result on injectivity of $\Ext$. Later we will see that this theorem can be viewed as a generalization of the Kodaira vanishing theorem.

\begin{theorem}
\label{thm.injectivityofExtforpuncturedDB}
Let $(R,\m)$ be a reduced Noetherian $\mathbb{N}$-graded $(R_0=\mathbb{C})$-algebra with $\m$ the unique homogeneous maximal ideal. Suppose $R_P$ is Du Bois for all $P\neq\m$. Write $R=A/I$ where $A=\mathbb{C}[x_1,\dots,x_n]$ is a polynomial ring with $\deg x_i=d_i>0$ and $I$ is a homogeneous ideal. Then the natural degree-preserving map $\Ext_A^j(R, A)\to H_I^j(A)$ induces an injection \[
\big[\Ext_A^j(R, A)\big]_{\geq -d}\hookrightarrow \big[H_I^j(A)\big]_{\geq -d}
\]
for every $j$, where $d=\sum d_i$.
\end{theorem}
\begin{proof}
We have the hypercohomology spectral sequence
\[
H_\m^p(h^q(\DuBois{R}))\Rightarrow \mathbb{H}_\m^{p+q}(\DuBois{R}).
\]
Since $R$ is Du~Bois away from $V(\m)$, $h^q(\DuBois{R})$ has finite length when $q\geq 1$. Thus we know \mbox{$H_\m^p(h^q(\DuBois{R}))=0$} unless $p=0$ or $q=0$.  We also have that $H_\m^0(h^i(\DuBois{R}))\cong h^i(\DuBois{R})$ for $i \geq 1$. Hence the above spectral sequence carries the data of a long exact sequence:
\begin{equation}
\label{equation LES}
{\def\arraystretch{1.6}
\begin{array}{cccccc}
0& \to & H_\m^1(h^0(\DuBois{R})) \to & \mathbb{H}_\m^{1}(\DuBois{R}) \to & h^1(\DuBois{R}) \to\\
 & \to & H_\m^2(h^0(\DuBois{R})) \to & \mathbb{H}_\m^{2}(\DuBois{R}) \to & h^2(\DuBois{R}) \to\\
& &\cdots&\cdots&\cdots\\
& \to & H_\m^i(h^0(\DuBois{R})) \to & \mathbb{H}_\m^{i}(\DuBois{R}) \to & h^i(\DuBois{R}) \to & \cdots
\end{array}
}
\end{equation}
We also have short exact sequence $0\to R\to h^0(\DuBois{R}) \rightarrow \big[H_\m^1(R)\big]_{>0}\to 0$ by \autoref{prop.CriterionforDB}. Therefore the long exact sequence of local cohomology and the observation that $\big[H^1_{\fram}(R)\big]_{>0}$ has finite length tells us that
\[
H_\m^1(h^0(\DuBois{R}))\cong H_\m^1(R)/\big[H_\m^1(R)\big]_{>0}
\] and $$H_\m^i(h^0(\DuBois{R}))\cong H_\m^i(R) \text{ for every } i\geq 2.$$
Now taking the degree $\leq0$ part of (\ref{equation LES}) and again using \autoref{prop.CriterionforDB} yields:
\begin{equation}
\label{equation isom in negative degree}
\big[H_\m^i(R)\big]_{\leq0}\cong \big[\mathbb{H}_\m^i(\DuBois{R})\big]_{\leq0} \text{ for every } i.
\end{equation}
At this point, note that by the Matlis dual of \autoref{key injectivity} (see the proof of \autoref{surjectivity in local cohomology} applied to $S=A/J$ for $\sqrt{J}=I$, and thus $S_{\red}=A/I=R$), we always have a surjection $H_\m^i(A/J)\twoheadrightarrow \mathbb{H}_\m^i(\DuBois{R})$ for every $i$ and $\sqrt{J}=I$. Taking the degree $\leq 0$ part and applying (\ref{equation isom in negative degree}), we thus get:
\begin{equation}
\label{equation surjectivity in nonpositive degree}
\big[H_\m^i(A/J)\big]_{\leq0}\twoheadrightarrow \big[H_\m^i(R)\big]_{\leq0}
\end{equation}
is surjective for every $i$ and $\sqrt{J}=I$. Now taking $J=I^t$ and applying graded local duality (we refer to \cite{BrodmannSharpLocalCohomology} for definitions and standard properties of graded canonical modules, graded injective hulls and graded local duality, but we emphasize here that the graded canonical module of $A$ is $A(-d)$), we have:
\[
\big[\Ext_A^j(R, A(-d))\big]_{\geq 0}\hookrightarrow \big[\Ext_A^j(A/I^t, A(-d))\big]_{\geq 0}
\]
is injective for every $j$ and $t$. So after taking a direct limit and a degree shift, we have
\[
\big[\Ext_A^j(R, A)\big]_{\geq -d}\hookrightarrow \big[H_I^j(A)\big]_{\geq -d}
\]
is injective for every $j$. This finishes the proof.
\end{proof}

\begin{remark}
\label{remark.dualform}
The dual form of \autoref{thm.injectivityofExtforpuncturedDB} says that $\big[H_\m^i(A/J)\big]_{t}\twoheadrightarrow \big[H_\m^i(R)\big]_{t}$ is surjective for every $i$, every $t\leq 0$ and every $\sqrt{J}=I$, see (\ref{equation surjectivity in nonpositive degree}). When $R=A/I$ has isolated singularity, $A$ is standard graded and $t=0$, this was proved in \cite[Proposition 3.8]{VarbaroCohomologicaldimension}. Therefore our \autoref{thm.injectivityofExtforpuncturedDB} greatly generalized this result.
\end{remark}

In general, we cannot expect $\big[\Ext_A^j(R, A)\big]_{< -d}\hookrightarrow \big[H_I^j(A)\big]_{< -d}$ is injective under the hypothesis of \autoref{thm.injectivityofExtforpuncturedDB} (even if $R$ is an isolated singularity). Consider the following example.

\begin{example}[{\it cf.} Example 3.5 in \cite{SinghandWaltherlocalcohomologyandpure}]
Let $R=\mathbb{C}[s^4, s^3t, st^3, t^4]$. Then we can write $R=A/I$ where $A=\mathbb{C}[x,y.z,w]$ with standard grading (i.e., $x,y,z,w$ all have degree one). It is straightforward to check that $R$ is an isolated singularity with
\[
H_\m^1(R)=\big[H_\m^1(R)\big]_{>0}\neq 0
\]
 (in particular $\depth R=1$). By graded local duality, we have $\big[\Ext_A^3(R,A)\big]_{<-4}\neq 0$. On the other hand, using standard vanishing theorems in \cite[Theorem 2.9]{HunekeLyubeznikVanishingLocalcohomology} we know that $H_I^3(A)=0$. Therefore the map $\big[\Ext_A^3(R,A)\big]_{<-4}\rightarrow \big[H_I^3(A)\big]_{<-4}$ is not injective.
\end{example}

An important consequence of \autoref{thm.injectivityofExtforpuncturedDB} (in fact we only need the injectivity in degree $>-d$) is the following vanishing result:

\begin{theorem}
\label{thm.vanishing of local cohomolog for DB}
Let $(R,\m)$ be a Noetherian $\mathbb{N}$-graded $(R_0=\mathbb{C})$-algebra with $\m$ the unique homogeneous maximal ideal. Suppose $R_P$ is Du Bois for all $P\neq\m$ and $H_\m^i(R)$ has finite length for some $i$. Then $\big[H_\m^i(R)\big]_{<0}=0$.

In particular, if $R$ is equidimensional and is Cohen-Macaulay Du Bois on $\Spec R-\{\m\}$, then $\big[H_\m^i(R)\big]_{<0}=0$ for every $i<\dim R$.
\end{theorem}
\begin{proof}
Let $R=A/I$ where $A=\mathbb{C}[x_1,\dots,x_n]$ is a (not necessarily standard graded) polynomial ring and $I$ is a homogeneous ideal. Set $\deg x_i=d_i>0$ and $d=\sum {d_i}$. The graded canonical module of $A$ is $A(-d)$. By graded local duality,
\[
\Ext_A^{n-i}(R, A)(-d)\cong H_\m^i(R)^*.
\]
Therefore if $\big[H_\m^i(R)\big]_{-j}\neq 0$ for some $j>0$, then $\big[\Ext_A^{n-i}(R, A)\big]_{j-d}\neq 0$ for some $j>0$. By \autoref{thm.injectivityofExtforpuncturedDB}, we have an injection:
\[
\big[\Ext_A^{n-i}(R, A)\big]_{j-d}\hookrightarrow \big[H_I^{n-i}(A)\big]_{j-d}.
\]
Since $H_\m^i(R)$ has finite length, $\Ext_A^{n-i}(R, A)$ also has finite length. Hence the natural degree-preserving map $\Ext_A^{n-i}(R, A)\to H_I^{n-i}(A)$ factors through $$\Ext_A^{n-i}(R, A)\to H_\m^0H_I^{n-i}(A) \to H_I^{n-i}(A).$$ Taking the degree $j-d$ part, we thus get an injection:
\[
\big[\Ext_A^{n-i}(R, A)\big]_{j-d}\rightarrow \big[H_\m^0H_I^{n-i}(A)\big]_{j-d}.
 \]
 It follows that $\big[H_\m^0H_I^{n-i}(A)\big]_{j-d}\neq 0$ for some $j>0$. However, $H_\m^0H_I^{n-i}(A)$ is an Eulerian graded $\scr{D}$-module supported only at $\m$. Thus by \cite[Theorem 1.2]{MaZhangEuleriangradedDmodules},\footnote{\cite[Theorem 1.2]{MaZhangEuleriangradedDmodules} assumes $A=\mathbb{C}[x_1,\dots,x_n]$ has standard grading, i.e., $d_i=1$ and hence $d=\sum d_i=n$. However, the same proof can be adapted to the general case: one only needs to replace the Euler operator $\sum x_i\partial_i$ by $\sum d_ix_i\partial_i$. The reader is referred to \cite[Section 2]{PuthenpurakalDeRhamlocalcohomology} for a discussion on this.} the socle of $H_\m^0H_I^{n-i}(A)$ is concentrated in degree $-d$ so that $\big[H_\m^0H_I^{n-i}(A)\big]_{>-d}= 0$, a contradiction.
\end{proof}

If $R$ is the section ring of a normal Cohen-Macaulay and Du Bois projective variety $X$ with respect to an ample line bundle $\scr{L}$, then $\big[H_\m^i(R)\big]_{<0}=0$ is exactly the Kodaira vanishing for $X$ (which is well-known, for example see \cite{MaF-injectivityandBuchsbaumsingularities} or \cite{PatakfalviSemiNegativityHodgeBundles}). But \autoref{thm.vanishing of local cohomolog for DB} can handle more general $R$, i.e., $R$ need not be a section ring of an ample line bundle.  If $R$ is normal, then any graded ring is the section ring of some $\bQ$-divisor \cite{DemazureNormalGradedRings} and in that case, our results yield variants and consequences of Kawamata-Viehweg vanishing (also see \cite[Lemma 2.1 and Proposition 2.2]{WatanabeRemarksonDemazure}).  But for general graded rings we do not know how to view them as section rings.  Thus our results \autoref{thm.injectivityofExtforpuncturedDB} and \autoref{thm.vanishing of local cohomolog for DB} should be viewed as generalizations of the Kodaira vanishing theorem for Cohen-Macaulay Du Bois projective varieties.  It would also be natural to try to generalize the results of \autoref{prop.CriterionforDB} through \autoref{thm.vanishing of local cohomolog for DB} to the context of Du~Bois pairs.  One particular obstruction is the use of the Eulerian graded $\scr{D}$-module at the end of the proof of \autoref{thm.vanishing of local cohomolog for DB}.

\subsection{Set-theoretic Cohen-Macaulayness} Our next application is a characteristic $0$ analog of \cite[Lemma 3.1]{SinghandWaltherArithmeticrankofSegreproducts} on set-theoretic Cohen-Macaulayness.  Recall that an ideal $I$ in a regular ring $R$ is \emph{set theoretically Cohen-Macaulay} if there exists an ideal $J$ such that $\sqrt{I} = \sqrt{J}$ and $R/J$ is Cohen-Macaulay.
\begin{proposition}
\label{prop.IdealNotSetTheoreticallyCM}
Let $(R,\m)$ be a regular local ring essentially of finite type over $\mathbb{C}$, and let $I\subseteq R$ be an ideal. If $R/I$ is Du Bois but not Cohen-Macaulay, then the ideal $I$ is not set-theoretically Cohen-Macaulay.
\end{proposition}
\begin{proof}
Suppose $I=\sqrt{J}$ for some $J$ such that $R/J$ is Cohen-Macaulay. Applying \autoref{surjectivity in local cohomology} to $S=R/J$, we find that for every $i<\dim R/I$, $$0=H_\m^i(R/J)\twoheadrightarrow H_\m^i(R/I)$$ is surjective and thus $R/I$ is Cohen-Macaulay, a contradiction.
\end{proof}


In the graded characteristic $0$ case, we have a stronger criterion for set-theoretic Cohen-Macaulayness.

\begin{corollary}
Let $R=\mathbb{C}[x_1,\dots,x_n]$ be a polynomial ring with possibly non-standard grading. Let $I$ be a homogeneous ideal of $R$ such that $R/I$ is Du Bois on $\Spec R-\{\m\}$ (e.g., $R/I$ has an isolated singularity at $\{\m\}$).

Suppose $\big[H_\m^i(R/I)\big]_{\leq0}\neq0$ for some $i<\dim R/I$ (e.g., $R/I$ is not Cohen-Macaulay on $\Spec R-\{\m\}$, or $H^i(X, \O_X)\neq 0$ for $X=\Proj R/I$). Then $I$ is not set-theoretically Cohen-Macaulay.
\end{corollary}
\begin{proof}
Suppose $I=\sqrt{J}$ for some $J$ such that $R/J$ is Cohen-Macaulay. Applying the dual form of \autoref{thm.injectivityofExtforpuncturedDB} (see \autoref{remark.dualform}), we get $$0=\big[H_\m^i(R/J)\big]_{\leq 0}\twoheadrightarrow \big[H_\m^i(R/I)\big]_{\leq 0}$$ is surjective for every $i<\dim R/I$. This clearly contradict our hypothesis.
\end{proof}

We point out the following example as an application.

\begin{example}
\label{example:segreproductofElliptic}
Let $k$ be a field and let $E_k\subseteq\mathbb{P}^2_k$ be a smooth elliptic curve over $k$. We want to study the defining ideal of the Segre embedding $E_k\times\mathbb{P}^1_k\subseteq\mathbb{P}^5_k$. We let $k[x_0,\dots,x_5]/I=A/I$ be this homogeneous coordinate ring. It is well known that $A/I$ is not Cohen-Macaulay.
\begin{enumerate}
\item $k$ has characteristic $p>0$ and $E_k$ is an ordinary elliptic curve. In this case it is well known that $A/I$ is $F$-pure, so \cite[Lemma 3.1]{SinghandWaltherArithmeticrankofSegreproducts} shows $I$ is not set-theoretically Cohen-Macaulay.
\item $k$ has characteristic $p>0$ and $E_k$ is supersingular. We want to point out that, at least when $k$ is $F$-finite, $I$ is still not set-theoretically Cohen-Macaulay. This answers a question in \cite[Remark 3.4]{SinghandWaltherArithmeticrankofSegreproducts}. Suppose there exists $J$ such that $\sqrt{J}=I$ and $A/J$ is Cohen-Macaulay. Let $e\gg0$ such that $I^{[p^e]}\subseteq J$. The composite of the Frobenius map on $A/I$ with the natural surjection $A/I\xrightarrow{F}A/I^{[p^e]}\twoheadrightarrow A/J$ makes $A/J$ a small Cohen-Macaulay algebra over $A/I$ (note that $k$, and hence $A$, is $F$-finite). However, by a result of Bhatt \cite[Example 3.11]{BhattOntheNonexistenceofsmallCMalgebra}, $A/I$ does not have any small Cohen-Macaulay algebra, a contradiction.
\item $k$ has characteristic $0$. In this case, it is easy to check using \autoref{prop.CriterionforDB} that $A/I$ is Du Bois. Hence \autoref{prop.IdealNotSetTheoreticallyCM} immediately shows $I$ is not set-theoretically Cohen-Macaulay. This example was originally obtained in \cite[Theorem 3.3]{SinghandWaltherArithmeticrankofSegreproducts} using reduction to characteristic $p>0$. Thus our \autoref{prop.IdealNotSetTheoreticallyCM} can be viewed as a vast generalization of their result.
\end{enumerate}
\end{example}

It is worth to point out that in \autoref{example:segreproductofElliptic}, we know $H_\m^i(A/J)\twoheadrightarrow H_\m^i(A/I)$ for every $i$ and every $\sqrt{J}=I$ in characteristic $0$ since $A/I$ is Du Bois. In characteristic $p>0$, we have $H_\m^i(A/J)\twoheadrightarrow H_\m^i(A/I)$ for every $i$ and every $\sqrt{J}=I$ when $E_k$ is ordinary, however it is straightforward to check that $H_\m^2(A/I^{[p]})\to H_\m^2(A/I)$ is not surjective (it is the zero map) when $E_k$ is supersingular. Therefore the surjective property proved in \autoref{surjectivity in local cohomology} does not pass to reduction mod $p\gg0$.


\subsection{Koszul cohomology versus local cohomology} Our last application in this section is a strengthening of the main result of \cite{HochsterRobertsFrobeniusLocalCohomology}. We start by proving a general result which is characteristic-free. The proof is inspired by \cite{SchenzelApplicationsOfDualizingComplexes}.

\begin{theorem}
\label{char free theorem in analog of Hochster-Roberts}
Let $R$ be a Noetherian $\mathbb{N}$-graded $k$-algebra where $k = R_0$ is an arbitrary field. Let $\m$ be the unique homogenous maximal ideal. If $H_\m^r(R)=\big[H_\m^r(R)\big]_0$ for every $r<n=\dim R$, then $\big[H^r(\underline{x}, R)\big]_0\cong H_\m^r(R)$ for every $r<n$ and every homogeneous system of parameters $\underline{x}=x_1,\dots,x_n$, where $H^r(\underline{x}, R)$ denotes the $r$-th Koszul cohomology of $\underline{x}$. In other words, it is not necessary to take a direct limit when computing the local cohomology.
\end{theorem}
\begin{proof}
We fix $\underline{x}=x_1,\dots,x_n$ a homogeneous system of parameters. Let $\deg x_i=d_i>0$. Consider the graded Koszul complex:
$$K_\mydot: 0\to R(-d_1-d_2-\cdots-d_n)\to\cdots\to\oplus R(-d_i)\to R\to 0.$$
After we apply $\Hom_R(-, R)$ we obtain the graded Koszul cocomplex:
$$K^\mydot: 0\to R\to \oplus R(d_i)\to \cdots\to R(d_1+d_2+\cdots+d_n)\to 0.$$
We note that $K^\mydot$ lives in cohomology degree $0,1,\dots,n$.

Let $\omega_R^\mydot$ be the graded normalized dualizing complex of $R$, thus $\omega_R=h^{-n}\omega_R^\mydot$ is the graded canonical module of $R$. Let $(-)^*=\Hom_R(-, {}^*E)$ where ${}^*E$ is the graded injective hull of $k$. We have a triangle
$$\omega_R[n]\to\omega_R^\mydot\to\tau_{>-n}\omega_R^\mydot\xrightarrow{+1}$$
Applying $\myR\Hom_R(K^\mydot,-)$ we get:
$$\myR\Hom_R(K^\mydot, \omega_R[n])\to \myR\Hom_R(K^\mydot, \omega_R^\mydot)\to \myR\Hom_R(K^\mydot, \tau_{>-n}\omega_R^\mydot)\xrightarrow{+1}$$
Applying $\Hom_R(-,{}^*E)$ and using graded local duality, we obtain:
$$\myR\Hom_R(K^\mydot, \tau_{>-n}\omega_R^\mydot)^*\to K^\mydot \to \myR\Hom_R(K^\mydot, \omega_R[n])^*\xrightarrow{+1}$$
Note that $\myR\Hom_R(K^\mydot, \omega_R[n])^*$ lives in cohomological degree $n, n+1,\dots, 2n$,
hence we obtain a graded isomorphism in the derived category:
\[
\tau^{<n}K^\mydot\cong\tau^{<n}\myR\Hom_R(K^\mydot, \tau_{>-n}\omega_R^\mydot)^*.
\]
Therefore for every $r<n$, we have:
$$h^r(K^\mydot)\cong h^r(\myR\Hom_R(K^\mydot, \tau_{>-n}\omega_R^\mydot)^*).$$

At this point, notice that $H_\m^r(R)=\big[H_\m^r(R)\big]_0$ for every $r<n$ implies $R$ is Buchsbaum by \cite[Theorem 3.1]{SchenzelApplicationsOfDualizingComplexes}. This means $\tau_{>-n}\omega_R^\mydot$ is quasi-isomorphic to a complex of graded $k$-vector spaces. Moreover, we know that all these graded vector spaces have degree $0$ because $H_\m^i(R)=\big[H_\m^i(R)\big]_0$ for every $i<n$. In other words, we have:
$$\tau_{>-n}\omega_R^\mydot\cong 0\to k^{s_{n-1}}\to k^{s_{n-2}}\to \cdots\to k^{s_1}\to k^{s_0}\to 0$$
where the complex on the right hand side has zero differentials, $k^{s_i}$ has internal degree $0$ and sits in cohomology degree $-i$, with $s_i=\dim_k\big[H_\m^i(R)\big]_0$.

Recall that $d_i>0$, hence by keeping track of the internal degrees we see that
$$\big[\myR\Hom_R(K^\mydot, \tau_{>-n}\omega_R^\mydot)\big]_0\cong\tau_{>-n}\omega_R^\mydot.$$
Now by graded local duality, we have:
$$\big[H^r(\underline{x}, R)\big]_0=\big[h^r(K^\mydot)\big]_0\cong \big[h^r(\myR\Hom_R(K^\mydot, \tau_{>-n}\omega_R^\mydot)^*)\big]_0\cong h^r((\tau_{>n}\omega_R^\mydot)^*)\cong H_\m^r(R)$$ for every $r<n$. This finishes the proof.
\end{proof}

Now we can prove the following extension of the main result of \cite{HochsterRobertsFrobeniusLocalCohomology}.
\begin{corollary}
\label{analogue of Hochster-Roberts}
Let $R$ be a Noetherian $\mathbb{N}$-graded $(R_0 = k)$-algebra with $\m$ the unique homogeneous maximal ideal. Suppose $R$ is equidimensional and Cohen-Macaulay on $\Spec R-\{\m\}$. Assume one of the following:
\begin{enumerate}
\item $k$ has characteristic $p>0$ and $R$ is $F$-injective.
\item $k$ has characteristic $0$ and $R$ is Du Bois.
\end{enumerate}
Then $\big[H^r(\underline{x}, R)\big]_0\cong H_\m^r(R)$ for every $r<n=\dim R$ and every homogeneous system of parameters $\underline{x}=x_1,\dots,x_n$, where $H^r(\underline{x}, R)$ denotes the $r$-th Koszul cohomology of $\underline{x}$. In other words, it is not necessary to take a direct limit when computing the local cohomology.
\end{corollary}
\begin{proof}
Since $R$ is equidimensional and Cohen-Macaulay on the punctured spectrum, we know that $H_\m^r(R)$ has finite length for every $r<n=\dim R$. We will show  $H_\m^r(R)=\big[H_\m^r(R)\big]_0$ for every $r<n$ in situation $(a)$ or $(b)$. This will finish the proof by \autoref{char free theorem in analog of Hochster-Roberts}.

In situation $(a)$, $H_\m^r(R)=\big[H_\m^r(R)\big]_0$ is obvious because $H_\m^r(R)$ has finite length and Frobenius acts injectively on it. In situation $(b)$, notice that $\big[H_\m^r(R)\big]_{<0}=0$ by \autoref{thm.vanishing of local cohomolog for DB} while $\big[H_\m^r(R)\big]_{>0}=0$ by \autoref{prop.CriterionforDB}, hence $H_\m^r(R)=\big[H_\m^r(R)\big]_0$.
\end{proof}

\begin{remark}
In situation $(b)$ above, if $R$ is normal standard graded, then $H_\m^r(R)=\big[H_\m^r(R)\big]_0$ also follows from \cite[Theorem 4.5]{MaF-injectivityandBuchsbaumsingularities}. However, in the above proof we don't need any normal or standard graded hypothesis thanks to \autoref{prop.CriterionforDB} and \autoref{thm.vanishing of local cohomolog for DB}.
\end{remark}



\begin{remark}
\autoref{analogue of Hochster-Roberts} was proved when $R$ is $F$-pure and $k$ is perfect in \cite[Theorem 1.1]{HochsterRobertsFrobeniusLocalCohomology}, and by a technical reduction to $p>0$ technique, it was also proved when $R$ is of $F$-pure type \cite[Theorem 4.8]{HochsterRobertsFrobeniusLocalCohomology}. Since $F$-pure certainly implies $F$-injective and $F$-injective  type implies Du Bois (see \cite{SchwedeFInjectiveAreDuBois}), our theorem gives a generalization of Hochster-Roberts's result, and our proof is quite different from that of \cite{HochsterRobertsFrobeniusLocalCohomology}.
\end{remark}

We end this section by pointing out an example showing that in \autoref{char free theorem in analog of Hochster-Roberts} or \autoref{analogue of Hochster-Roberts}, it is possible that $H^r(\underline{x}, R)\neq H_\m^r(R)$, i.e., we must take the degree 0 piece of the Koszul cohomology. This is a variant of \autoref{example:segreproductofElliptic}.

\begin{example}
Let $R=\frac{k[x,y,z]}{x^3+y^3+z^3}\# k[a,b,c]$ be the Segre product
of $\frac{k[x,y,z]}{x^3+y^3+z^3}$ and $k[a,b,c]$ where the characteristic of $k$ is either 0 or congruent to $1$ mod $3$. Therefore $R$ is the homogeneous coordinate ring of the Segre embedding of $X=E\times \mathbb{P}^2$ to $\mathbb{P}^8$, where $E=\Proj \frac{k[x,y,z]}{x^3+y^3+z^3}$ is an elliptic curve. Notice that $\dim X=3$ and $\dim R=4$. Since $R$ has an isolated singularity it is Cohen-Macaulay on the punctured spectrum. It is easy to check that $R$ is Du Bois in characteristic $0$, and $F$-pure (and thus $F$-injective) in characteristic $p>0$ since $p\equiv 1$ mod $3$. So we know that $H_\m^i(R)=[H_\m^i(R)]_0$ for every $i<4$ where $\m$ is the unique homogeneous maximal ideal of $R$. Now we compute:
\[H_\m^2(R)=[H_\m^2(R)]_0=H^1(X, \O_X)=\oplus_{i+j=1}H^i(E, \O_E)\otimes_kH^j(\mathbb{P}^2, \O_{\mathbb{P}^2})\cong k\]
\[H_\m^3(R)=[H_\m^3(R)]_0=H^2(X, \O_X)=\oplus_{i+j=2}H^i(E, \O_E)\otimes_kH^j(\mathbb{P}^2, \O_{\mathbb{P}^2})=0 .\]
The first line shows that $R$ is not Cohen-Macaulay (actually we have $\depth R=2$), so for any homogeneous system of parameters $\underline{x}$ of $R$, we have $H^3(\underline{x}, R)\cong H_1(\underline{x}, R)\neq 0$. Hence $H^3(\underline{x}, R)\neq H_\m^3(R)$.
\end{example}

\section{Deformation of dense $F$-injective type}

In this section we use results in section 3 to prove that singularities of (dense) $F$-injective type in characteristic $0$ deform. In fact, we prove a slightly stronger result. Our motivation for studying this question is a recent result of Kov\'{a}cs-Schwede \cite{KovacsSchwedeDBDeforms} that Du Bois singularities in characteristic $0$ deform.

By the main result of \cite{SchwedeFInjectiveAreDuBois}, singularities of dense $F$-injective type are always Du Bois. It is conjectured \cite[Conjecture 4.1]{BhattSchwedeTakagiweakordinaryconjectureandFsingularity} that Du Bois singularities should be equivalent to singularities of dense $F$-injective type. This conjecture is equivalent to the weak ordinarity conjecture of Musta\c{t}\u{a}-Srinivas \cite{MustataSrinivasOrdinary}. If this conjecture is true, then singularities of dense $F$-injective type deform because Du Bois singularities deform \cite{KovacsSchwedeDBDeforms}. However, the weak ordinarity conjecture is wide open.

\begin{setup}[Reduction to characteristic $p > 0$]
\label{setup.ReductionToCharP}
We recall briefly reduction to characteristic $p > 0$.  For more details in our setting, see \cite[Section 2.1]{HochsterHunekeTightClosureInEqualCharactersticZero} and \cite[Section 6]{SchwedeFInjectiveAreDuBois}.

Suppose $(R, \m)$ is essentially of finite type over $\bC$ (or another field of characteristic zero, which we will also call $\bC$) so that $(R,\m)$ is a homomorphic image of $T_P$ where $T=\bC[x_1,\dots,x_t]$ and $P \subseteq T$ is a prime ideal so that $R = (T/J)_P$.  Given a finite collection of finitely generated $R$-modules $M_i$ (and finitely many maps between them), we may assume that each $M_i = (M_i')_P$ for a finitely generated $T$-module $M_i'$ annihilated by $J$.  Suppose $E \to \Spec T/J$ is the reduced preimage of $T/J$ in a log resolution of $(\Spec T, \Spec T/J)$.  We also keep track of $E \to \Spec T/J$ in the reduction to characteristic $p > 0$ process, as well as the modules $h^i(\DuBois{T/J}) = \myR^i \pi_* \O_E$.  Assume $x$ is the image of $h(x_1, \ldots, x_t) \in T$ (note that if $x$ has a problematic denominator, we can replace $x$ by another element that generates the same ideal and does not have a denominator).

We pick a finitely generated regular $\mathbb{Z}$-algebra $A\subseteq\mathbb{C}$ such that the coefficients of the generators of $P$, the coefficients of $h$, the coefficients of $J$, and the coefficients of a presentation of the $M_i'$ are contained in $A$.  Form $T_A=A[x_1,\dots,x_t]$ and let $P_A = P \cap T_A, J_A = J \cap T_A$ and observe that $P_A \tensor_A \bC = P, J_A \tensor_A \bC = J$ by construction.  Note that by generic flatness, we can replace $A$ by $A[b^{-1}]$ and so shrink $\Spec A$ if necessary to assume that $R_A$ is a flat $A$-module and $x$ is a nonzerodivisor on $R_A$. Likewise form $(M_i')_A$ with the same presentation matrix of $M_i'$ (and likewise with maps between the modules) and form $E_A \to \Spec T_A/J_A$ so that $(M_i')_A \otimes_A \bC = M_i'$ and that $(\myR^i \pi_* \O_{E_A}) \otimes_A \bC = (\myR^i \pi_* \O_{E_A})$.  Shrinking $\Spec A$ yet again if necessary, we can assume all these modules are flat over $A$ (and that any relevant kernels and cokernels of the finitely many maps between them are also flat).

We now mod out by a maximal ideal $\n$ of $A$ with $\kappa = A/\n$.  In particular, we use $R_A = (T_A/J_A)_{P_A}$, etc. and $R_\kappa$, $T_\kappa$, $E_\kappa$ etc. to denote the corresponding rings, schemes and modules over $A$ and $\kappa$, where $\kappa=A/\n$ for $\n$ a maximal ideal in $A$. Since all relevant kernels and cokernels of the maps are flat by shrinking $\Spec A$, $x$ is a nonzerodivisor on $R_\kappa$ for every $\kappa$.
\end{setup}

We also need a slightly different version of the main result of \cite{SchwedeFInjectiveAreDuBois}.  In that paper, it was assumed that if $R$ is finite type over a field of characteristic zero, and of dense $F$-injective type, then $R$ has Du~Bois singularities.  We need a version of this result in which $R$ is local.  The reason is, that, in the notation of \autoref{thm.FInjectiveTypeDeforms}, if $R$ is of finite type, and $R_p/x_p R_p$ is $F$-injective, then for $p$ sufficiently divisible, we obtain that $R_p$ is $F$-injective in a neighborhood of $V(x_p)$.  The problem is that we do not know to control these neighborhoods as $p$ varies.  Thus we need the following preliminary result.  In particular, we do not know whether having dense $F$-injective type is an open property in characteristic zero.

\begin{theorem}
\label{thm.DBdenseF-injectivetypeLocal}
Let $(R,\m)$ be a local ring essentially of finite type over $\mathbb{C}$ and suppose that $R$ has dense $F$-injective type.  Then $R$ is Du~Bois.
\end{theorem}
The strategy is the same as in \cite{SchwedeFInjectiveAreDuBois}, indeed, the proof only differs in how carefully we keep track of a minimal prime of the non-Du~Bois locus.
\begin{proof}[Sketch of the proof]
It is easy to see that $R$ is seminormal so we need to show that $\myH^i({\DuBois{R}}) = 0$ for all $i > 0$.   We use the notation of \autoref{setup.ReductionToCharP}. Let $Q_{\bC} \subseteq P \subseteq T/J$ correspond to a minimal prime of the non-Du~Bois locus of $(R, \m)$, so that $\myH^i(\DuBois{(T/J)_{Q_{\bC}}}) = \myH^i(\myR \pi_* \O_{E})_{Q_{\bC}}$ has finite length for $i > 0$.  Since $(R, \m)$ has dense $F$-injective type, it is easy to see that so does $(R_{Q_{\bC}}, Q_{\bC} R_{Q_{\bC}})$ so we may assume that $\m = Q_{\bC}$.  Now using \autoref{setup.ReductionToCharP}, reduce to some model so that $(R_{\kappa}, \m_{\kappa})$ is $F$-injective.  The proof now follows exactly that of \cite[Proof 6.1]{SchwedeFInjectiveAreDuBois} where we obtain that
\[
\myH^i(\myR (\pi_{\kappa})_* \O_{E_{\kappa}})_{Q_{\kappa}} \hookrightarrow H^{i+1}_{Q_{\kappa}}( R_{\kappa} ).
\]
But the left side is annihilated by a power of Frobenius by \cite[Theorem 7.1]{SchwedeFInjectiveAreDuBois} and Frobenius acts injectively on the right by hypothesis.  The result follows.
\end{proof}

\begin{theorem}
\label{thm.FInjectiveTypeDeforms}
Let $(R,\m)$ be a local ring essentially of finite type over $\mathbb{C}$ and let $x$ be a nonzerodivisor on $R$. Suppose $R/xR$ has dense $F$-injective type. Then for infinitely many $p>0$, the Frobenius action $x^{p-1}F$ on $H_{\m_p}^i(R_p)$ is injective for every $i$, where $(R_p,\m_p)$ is the reduction mod $p$ of $R$. In particular, $R$ has dense $F$-injective type.
\end{theorem}
\begin{proof}
By \autoref{thm.DBdenseF-injectivetypeLocal}, $R/xR$ has Du Bois singularities.  By \autoref{theorem.SheafversionSurj} (taking $Z = \emptyset$), $H^i_{\m}(R) \xrightarrow{\cdot x} H^{i}_{\m}(R)$ surjects for all $i$.  By Matlis duality, $\Ext^i_T(R, T) \xrightarrow{\cdot x} \Ext^i(R, T)$ injects for all $i$.  Spreading this out to $A$, and possibly inverting an element of $A$, we see that $\Ext^i_{T_A}(R_A, T_A) \xrightarrow{\cdot x} \Ext^i_{T_A}(R_A, T_A)$ injects for all $i$ (note there are only finitely many $\Ext$ to consider).  Inverting another element of $A$ if necessary, we deduce that
\[
\Ext^i_{T_\kappa}(R_\kappa, T_\kappa) \xrightarrow{\cdot x} \Ext^i_{T_\kappa}(R_\kappa, T_\kappa)
\]
injects for every $i$ and each maximal ideal $\frn \subseteq A$, setting $\kappa = A/\frn$.  We abuse notation and let $x$ denote the image of $x$ in $R_{\kappa}$.  Applying Matis duality again and considering the Frobenius on the local cohomology modules, we have the following collection of short exact sequences for each $i$.
\[
\xymatrix{
0 \ar[r] & H_{\m_\kappa}^i(R_\kappa/xR_\kappa) \ar[r] \ar[d]^F & H_{\m_\kappa}^{i+1}(R_\kappa) \ar[r]^{\cdot x} \ar[d]^{x^{p-1}F} & H_{\m_\kappa}^{i+1}(R_\kappa)\ar[r]\ar[d]^F & 0 \\
0 \ar[r] & H_{\m_\kappa}^i(R_\kappa/xR_\kappa) \ar[r]  & H_{\m_\kappa}^{i+1}(R_\kappa) \ar[r]^{\cdot x}  & H_{\m_\kappa}^{i+1}(R_\kappa)\ar[r] & 0
}
\]
where $p$ is the characteristic of $\kappa$ and $F$ denotes the natural Frobenius action on $H_{\m_\kappa}^i(R_\kappa/xR_\kappa)$ and $H_{\m_\kappa}^{i+1}(R_\kappa)$.

At this point recall that $R/xR$ has dense $F$-injective type. It follows that for infinitely many $p>0$, if the residue field $\kappa$ of $A$ has characteristic $p>0$ then the natural Frobenius action on $H_{\m_\kappa}^i(R_\kappa/xR_\kappa)$ is injective for every $i$. Now chasing the above diagram, if $x^{p-1}F$ is not injective, then we can pick $0\neq y\in \soc(H_{\m_\kappa}^{i+1}(R_\kappa))\cap \ker(x^{p-1}F)$. Since $y$ is in the socle of $H_{\m_\kappa}^{i+1}(R_\kappa)$, it maps to zero under multiplication by $x$. But then $0\neq y\in H_{\m_\kappa}^i(R_\kappa/xR_\kappa)$, and chasing the diagram we find that $x^{p-1}F(y)\neq 0$, which is a contradiction.

We have established that for infinitely many $p$, after we do reduction to $p$, the Frobenius action $x^{p-1}F$ on $H_{\m_\kappa}^{i+1}(R_\kappa)$ is injective for every $i$. This certainly implies the natural Frobenius action $F$ on $H_{\m_\kappa}^{i+1}(R_\kappa)$ is injective for every $i$. Hence $R$ has dense $F$-injective type.
\end{proof}

\begin{remark}
It is still unknown whether $F$-injective singularities in characteristic $p>0$ deform (and this has been open since \cite{FedderFPureRat}). Our theorem is in support of this conjecture: it shows this is true ``in characteristic $p \gg 0$". For the most recent progress on deformation of $F$-injectivity in characteristic $p>0$, we refer to \cite{HoriuchiMillerShimomoto}.
On the other hand, \autoref{thm.FInjectiveTypeDeforms} provides evidence for the weak ordinarity conjecture \cite{MustataSrinivasOrdinary} because of the relation between the weak ordinarity conjecture and the conjecture that Du~Bois singularities have dense $F$-injective type \cite{BhattSchwedeTakagiweakordinaryconjectureandFsingularity}.
\end{remark}

\bibliographystyle{skalpha}
\bibliography{MainBib}
\end{document}